\documentclass[11pt]{amsart}
\usepackage{amssymb,latexsym}
\usepackage{pdfsync}
\usepackage{color}
\usepackage[colorlinks,linkcolor=blue,citecolor=red]{hyperref}

\newdimen\AAdi%
\newbox\AAbo%
\def\AAk#1#2{\s_etbox\AAbo=\hbox{#2}\AAdi=\wd\AAbo\kern#1\AAdi{}}%
\def\AAr#1#2#3{\s_etbox\AAbo=\hbox{#2}\AAdi=\ht\AAbo\raise#1\AAdi\hbox{#3}}%
\font\tenmsb=msbm10 at 12pt
\font\sevenmsb=msbm7 at 8pt
\font\fivemsb=msbm5 at 6pt
\newfam\msbfam
\textfont\msbfam=\tenmsb
\scriptfont\msbfam=\sevenmsb
\scriptscriptfont\msbfam=\fivemsb
\def\Bbb#1{{\tenmsb\fam\msbfam#1}}

\newcommand{\beq}{\begin{equation}}
\newcommand{\eeq}{\end{equation}}
\newcommand{\beqr}{\begin{eqnarray}}
\newcommand{\eeqr}{\end{eqnarray}}
\newcommand{\ba}{\begin{array}}
\newcommand{\ea}{\end{array}}

\begin{document}

\newtheorem{thm}{Theorem}
\newtheorem{lem}{Lemma}
\newtheorem{cor}{Corollary}
\newtheorem{rem}{Remark}
\newtheorem{pro}{Proposition}
\newtheorem{defi}{Definition}
\newtheorem{eg}{Example}
\newtheorem*{claim}{Claim}
\newtheorem{conj}[thm]{Conjecture}
\newcommand{\noi}{\noindent}
\newcommand{\dis}{\displaystyle}
\newcommand{\mint}{-\!\!\!\!\!\!\int}
\numberwithin{equation}{section}

\def \bx{\hspace{2.5mm}\rule{2.5mm}{2.5mm}}
\def \vs{\vspace*{0.2cm}}
\def\hs{\hspace*{0.6cm}}
\def \ds{\displaystyle}
\def \p{\partial}
\def \O{\Omega}
\def \o{\omega}
\def \b{\beta}
\def \m{\mu}
\def \l{\lambda}
\def\L{\Lambda}
\def \ul{u_\lambda}
\def \D{\Delta}
\def \d{\delta}
\def \k{\kappa}
\def \s{\sigma}
\def \e{\varepsilon}
\def \a{\alpha}
\def \tf{\tilde{f}}
\def\cqfd{%
\mbox{ }%
\nolinebreak%
\hfill%
\rule{2mm} {2mm}%
\medbreak%
\par%
}
\def \pr {\noindent {\it Proof.} }
\def \rmk {\noindent {\it Remark} }
\def \esp {\hspace{4mm}}
\def \dsp {\hspace{2mm}}
\def \ssp {\hspace{1mm}}

\def\la{\langle}\def\ra{\rangle}

\def \u{u_+^{p^*}}
\def \ui{(u_+)^{p^*+1}}
\def \ul{(u^k)_+^{p^*}}
\def \energy{\int_{\R^n}\u }
\def \sk{\s_k}
\def \mo{\mu_k}
\def\cal{\mathcal}
\def \I{{\cal I}}
\def \J{{\cal J}}
\def \K{{\cal K}}
\def \OM{\overline{M}}

\def\n{\nabla}

\def\fk{{{\cal F}}_k}
\def\M1{{{\cal M}}_1}
\def\Fk{{\cal F}_k}
\def\Fl{{\cal F}_l}
\def\FF{\cal F}
\def\Gk{{\Gamma_k^+}}
\def\n{\nabla}
\def\uuu{{\n ^2 u+du\otimes du-\frac {|\n u|^2} 2 g_0+S_{g_0}}}
\def\uuug{{\n ^2 u+du\otimes du-\frac {|\n u|^2} 2 g+S_{g}}}
\def\sku{\sk\left(\uuu\right)}
\def\qed{\cqfd}
\def\vvv{{\frac{\n ^2 v} v -\frac {|\n v|^2} {2v^2} g_0+S_{g_0}}}
\def\vvs{{\frac{\n ^2 \tilde v} {\tilde v}
 -\frac {|\n \tilde v|^2} {2\tilde v^2} g_{S^n}+S_{g_{S^n}}}}
\def\skv{\sk\left(\vvv\right)}
\def\tr{\hbox{tr}}
\def\pO{\partial \Omega}
\def\dist{\hbox{dist}}
\def\RR{\Bbb R}\def\R{\Bbb R}
\def\C{\Bbb C}
\def\B{\Bbb B}
\def\N{\Bbb N}
\def\Q{\Bbb Q}
\def\Z{\Bbb Z}
\def\PP{\Bbb P}
\def\EE{\Bbb E}
\def\F{\Bbb F}
\def\G{\Bbb G}
\def\H{\Bbb H}
\def\SS{\Bbb S}\def\S{\Bbb S}

\def\div{\hbox{div}\,}

\def\lcf{{locally conformally flat} }

\def\circledwedge{\setbox0=\hbox{$\bigcirc$}\relax \mathbin {\hbox
to0pt{\raise.5pt\hbox to\wd0{\hfil $\wedge$\hfil}\hss}\box0 }}

\def\sss{\frac{\s_2}{\s_1}}

\date{\today}
\title[ Shrinking Sasaki-Ricci Solitons ]{ Gradient Shrinking Sasaki-Ricci
Solitons with Harmonic Weyl Tensor}
\author{}
\author[Chang]{Shu-Cheng Chang}
\address{ Department of Mathematics, National Taiwan University,
Taipei 10617, Taiwan and Shanghai Institute of Mathematics and
Interdisciplinary Sciences, Shanghai, 200433, China }
\email{scchang@math.ntu.edu.tw}
\author[Qiu]{Hongbing Qiu}
\address{School of Mathematics and Statistics\\
Wuhan University\\
Wuhan 430072, China }
\email{hbqiu@whu.edu.cn}



\begin{abstract}  We establish integral curvature estimates for complete gradient shrinking Sasaki-Ricci solitons. As an application, we show that any such soliton with harmonic Weyl tensor must be a finite quotient of a sphere. This result can be regarded as the Sasaki analogue of the work of Munteanu and Sesum \cite{MS13} on Ricci solitons.

\vskip12pt



\end{abstract}
\maketitle
\section{Introduction}




\vskip12pt

Sasakian geometry is as the odd-dimensional analogue of K\"{a}hler geometry.
A Riemannian $(2n+1)$-manifold $M$ is Sasaki if the cone $C(M)$ over $M$  is
K\"{a}hler cone. Furthermore, it is Sasaki-Einstein if its K\"{a}hler cone $
C(M)$ is a Calabi--Yau $(n+1)$-fold. In particular, Sasaki-Einstein $5$-manifolds provide interesting examples of the AdS/CFT correspondence. In
K\"ahler geometry, there is a well-known classification of compact Fano K\"{a}hler-Einstein smooth surfaces due to Tian--Yau and then leads to a first
classification of all compact regular Sasaki-Einstein $5$-manifolds. On the
other hand, it was proved by Smale--Barden (\cite{Sm62, Ba65}) that the
class of simply connected, closed, oriented, smooth, $5$-manifolds is
classifiable under diffeomorphism. Then it is possble to work on existence
problems of Sasaki-Einstein metrics on Sasakian manifolds of dimension five.
\emph{\ }We refer to \cite{BG08} and references therein for other examples of
(quasi-regular and irregular) Sasaki-Einsteins.

It has been observed, beginning with the work of Cao \cite{Cao85}, that the Kähler-Ricci flow can be effective in finding Kähler-Einstein metrics.
From this point of view, Smoczyk--Wang--Zhang (\cite{SWZ10}) study the
Sasaki-Ricci flow 
\begin{equation*}
\frac{d}{dt}g^{T}(x,t)=-Ric^{T}(x,t)
\end{equation*}
on $M\times \lbrack 0,T)$ and proved an existence theorem of Sasaki $\eta$-Einstein metrics on a compact Sasakian $(2n+1)$-manifold when the basic
first Chern class is positive or null. In general, the Sasaki-Ricci flow
will develop singularities in a finite time. we refer to \cite{CJ15}, \cite{CHLW25}, \cite{CLLW24} and \cite{CCLW22} for subsequent developments along this
direction.

In the present paper, we try to study singularities models of the
Sasaki-Ricci flow and classify the Sasaki-Ricci soliton $(M,g^{T},\psi ,X)$
which arises as a special solution to the flow and can be viewed as a natural generalization of $
\eta$-Einstein metric. In particular, we obtain a classification of complete gradient
shrinking Sasaki-Ricci solitons with the harmonic Weyl tensor.

We call $(M,g^{T},\psi ,X)$ a gradient Sasaki-Ricci solitons (or a
transverse K\"{a}hler-Ricci soliton) if there exist a Hamiltonian basic function 
$\psi $ and a transverse K\"{a}hler metric $g^{T}=(g_{j\overline{k}})$
such that 
\begin{equation*}
R_{j\overline{k}}^{T}+\psi _{j\overline{k}}=(A+2)g_{j\overline{k}}.
\end{equation*}
The soliton is said to be expanding, steady and shrinking if for a constant $A<-2;$\ \ $
A=-2;$ \ $A>-2$, respectively. Up to $D$-homothetic, one can take $A=2n$ such that 
\begin{equation}
R_{j\overline{k}}^{T}+\psi _{j\overline{k}}=(2n+2)g_{j\overline{k}}
\label{1}
\end{equation}
is a gradient shrinking Sasaki-Ricci soliton. In the special case where $\psi$ is constant, this reduces to the transverse K\"ahler-Einstein condition. Further details are provided in Section 2.

We first recall the results of the classification of gradient shrinking
Ricci solitons with vanishing Weyl tensor. In particular, Ni--Wallach (\cite{NW08}), Petersen--Wylie (\cite{PW10}), Zhang (\cite{Zh09}) and Cao--Wang--Zhang (\cite{CWZ11}) proved various classification theorems for complete gradient shrinking Ricci
solitons with vanishing Weyl curvature tensor in arbitrary dimension, under certain
integral curvature estimates. Subsequently,  Munteanu--Sesum (\cite{MS13}) removed these integral curvature assumptions and obtained a full classification of
complete gradient shrinking Ricci solitons with harmonic Weyl tensor.

Now we state our main result. Firstly, we establish an integral bound of the Ricci curvature for any gradient shrinking Sasaki-Ricci soliton as follows:

\begin{thm}\label{thm-SGR1}

Let $(M^{2n+1}, g^T, \psi, X)$ be a complete gradient shrinking Sasaki-Ricci soliton with $\psi$ the corresponding real Hamiltonian basic function (potential) with respect to $X$, then we have
\begin{equation}\label{eqn-SR20}\aligned
\int_M | Ric |^2 e^{-\l \psi} < \infty, \quad \quad for \quad any \quad \l>0.
\endaligned
\end{equation}

\end{thm}

Next by combining Lemma \ref{lem1} with Theorem \ref{thm-SGR1}, we prove that the identity (\ref{eqn-SR20'}), which is crucial in our classification result (see Theorem \ref{thm-main}),  holds under a weighted $L^2$-bound of the Riemannian curvature tensor condition.

\begin{thm}\label{thm-SGR2}

Let $(M^{2n+1}, g^T, \psi, X)$ be a complete gradient shrinking Sasaki-Ricci soliton with $\psi$ the corresponding real Hamiltonian basic function (potential) with respect to $X$.  Suppose that for some $\lambda < 1$, 
$
\int_M |Rm|^2 e^{-\l \psi} < \infty.
$
Then the following estimate holds
\begin{equation}\label{eqn-SR20'}\aligned
\int_M |{\n Ric}|^2 e^{-\psi} = \int_M |{\rm div} Rm|^2 e^{-\psi} < \infty.
\endaligned
\end{equation}

\end{thm}

Consequently, we can show that the above identity (\ref{eqn-SR20'}) holds for gradient shrinking Sasaki-Ricci solitons with harmonic Weyl tensor (see Corollary \ref{cor-SGR1}).  
  Moreover, by using the idea of \cite{MS13} and \cite{FG11}, we obtain the following classification for gradient shrinking Sasaki-Ricci solitons with harmonic Weyl tensor. This result may be regarded as the Sasakian counterpart of the classification for gradient shrinking Ricci solitons in \cite{MS13}. 

\begin{thm}\label{thm-main}

Let $(M^{2n+1}, g^T, \psi, X)$ be a complete gradient shrinking Sasaki-Ricci soliton with harmonic Weyl tensor.
  Then $(M^{2n+1}, g^T, \psi, X)$ is a finite quotient of $\mathbb S^{2n+1}$.

\end{thm} 

\vskip24pt

\section{Preliminaries}

In this section, we will recall some fundamental notions and identities for Sasakian manifolds and Sasaki-Ricci solitons. The reader is referred to \cite{BG08, CLL25, FOW09} and the references therein for some details.

\begin{defi}\label{def1}
A $(2n+1)$-dimensional Riemannian manifold $(M^{2n+1},g)$ is called a Sasaki
manifold if the cone manifold \ 
\begin{equation*}
(C(M),\overline{g},J)=(M\times \mathbb{R}^{+}\mathbf{,}r^{2}g+dr^{2},J)
\end{equation*}%
is K\"{a}hler. Note that $\{r=1\}=\{1\}\times M\subset C(M)$. Define the
Reeb vector field 
\begin{equation*}
\begin{array}{c}
\xi =J(\frac{\partial }{\partial r})%
\end{array}%
\end{equation*}%
and the contact $1$-form 
\begin{equation*}
\eta (Y)=g(\xi ,Y).
\end{equation*}%
Then $\eta (\xi )=1$ and $d\eta (\xi ,X)=0.$ $\xi $ is killing with unit
length. Furthermore, there is a natural splitting 
\begin{equation*}
\begin{array}{c}
TC(M)=L_{r\frac{\partial }{\partial r}}\oplus L_{\xi }\oplus H.%
\end{array}%
\end{equation*}%
Choose $JY=\Phi (Y)-\eta (Y)r\frac{\partial }{\partial r}$ and $\Phi
^{2}=-I+\eta \otimes \xi .$ We have $g(\Phi X,\Phi Y)=g(X,Y)-\eta (X)\eta (Y)
$ which is\ 
\begin{equation*}
g=g^{T}+\eta \otimes \eta .
\end{equation*}%
Here $(M,\xi ,\eta ,g,\Phi )$ is called the Sasakian structure.
\end{defi}

Let $\{U_{\alpha }\}_{\alpha \in A}$ be an open covering of the Sasakian
manifold $(M,\xi ,\eta ,g,\Phi )$ and $\pi_{\alpha }: U_{\alpha}\rightarrow
V_{\alpha }\subset \mathbb{C}^{n}$ submersion such that $\pi _{\alpha }\circ \pi _{\beta }^{-1}:\pi
_{\beta }(U_{\alpha }\cap U_{\beta })\rightarrow \pi _{\alpha }(U_{\alpha
}\cap U_{\beta })$ is biholomorphic. On each $V_{\alpha },$ there is a
canonical isomorphism $d\pi _{\alpha }:D_{p}\rightarrow T_{\pi _{\alpha
}(p)}V_{\alpha }$ for any $p\in U_{\alpha },$ where $D=\ker \eta \subset TM.$
Since $\xi $ generates isometrics, the restriction of the Sasakian metric $g$
to $D$ gives a well-defined Hermitian metric $g_{\alpha }^{T}$ on $V_{\alpha
}.$ This Hermitian metric in fact is K\"{a}hler. Then the K\"{a}hler $2$%
-form $\omega _{\alpha }^{T}$ of the Hermitian metric $g_{\alpha }^{T}$ on $%
V_{\alpha },$ which is the same as the restriction of the Levi form $d\eta $
to $\widetilde{D_{\alpha }^{n}}$, the slice $\{x=$ \textrm{constant}$\}$ in $%
U_{\alpha },$ is closed. The collection of K\"{a}hler metrics $\{g_{\alpha
}^{T}\}$ on $\{V_{\alpha }\}$ is so-called a transverse K\"{a}hler metric. 

For $\widetilde{X},\widetilde{Y},\widetilde{W},\widetilde{Z}\in \Gamma (TD)$
and the $d\pi _{\alpha }$-corresponding $X,Y,W,Z\in \Gamma (TV_{\alpha }).$
The Levi-Civita connection $\nabla _{X}^{T}$ with respect to the transverse K%
\"{a}hler metric $g^{T}:$%
\begin{equation*}
\begin{array}{rcl}
\nabla _{X}^{T}Y & := & d\pi _{\alpha }(\nabla _{\widetilde{X}}\widetilde{Y}%
), \\ 
\widetilde{\nabla _{X}^{T}Y} & := & \nabla _{\widetilde{X}}\widetilde{Y}%
+g(JX,Y)\xi , \\ 
Rm^{T}(X,Y,Z,W) & = & Rm_{D}(\widetilde{X},\widetilde{Y},\widetilde{Z},%
\widetilde{W})+g(J\widetilde{Y},\widetilde{W})g(J\widetilde{X},\widetilde{Z})
\\ 
&  & -g(J\widetilde{X},\widetilde{W})g(J\widetilde{Y},\widetilde{Z})-2g(J%
\widetilde{X},\widetilde{Y})g(J\widetilde{Z},\widetilde{W}), \\ 
Ric^{T}(X,Z) & = & Ric_{D}(\widetilde{X},\widetilde{Z})+2g(\widetilde{X},%
\widetilde{Z}).%
\end{array}%
\end{equation*}
 We also have
\begin{equation}\label{eqn-Rm}
R(X,\xi )Y=g(\xi ,Y)X-g(X,Y)\xi .
\end{equation}%
Then with respective to the transverse Levi-Civita connection $\nabla ^{T}$ 
\begin{equation}\label{eqn-SR2}
Ric^{T}=Ric+2g^{T}
\end{equation}%
on $D=\ker \eta $.

\begin{defi}
A Sasakian manifold $(M,\xi,\eta,g,\Phi)$ is $\eta$-Einstein if there is a
constant $A$ such that the Ricci curvature 
\begin{equation*}
Ric=Ag+(2n-A)\eta\otimes\eta. 
\end{equation*}
\end{defi}

For $g=g^{T}+\eta \otimes \eta $, we have 
\begin{equation*}
Ric^{T}=(A+2)g^{T}+2n(\eta \otimes \eta )
\end{equation*}%
and 
\begin{equation*}
Ric=Ag^{T}+2n(\eta \otimes \eta ).
\end{equation*}%
For $A=2n$%
\begin{equation*}
Ric=2ng^{T}+2n(\eta \otimes \eta )=2ng
\end{equation*}%
which is Sasaki-Einstein. In particular, $(M,g)$ is Sasaki-Einstein with 
\begin{equation*}
Ric_{g}=2ng
\end{equation*}%
if and only if the K\"{a}hler cone $(C(M),\overline{g})$ is Ricci-flat. On
the other hand, the transverse K\"{a}hler structure to the Reeb foliation $%
F_{\xi }$ is K\"{a}hler-Einstein with 
\begin{equation*}
Ric^{T}=2(n+1)g^{T}.
\end{equation*}

\begin{defi}
(\cite{FOW09, CLL25, BW08, Cao10, Ham98}) $(M,g^{T},\psi ,X)$ is called a Sasaki-Ricci
soliton (with the Hamiltonian potential $\psi $) with respect to the
Hamiltonian holomorphic vector field $X$ if, for a basic function $\psi $
such that 
\begin{equation}
\begin{array}{c}
\left\{ 
\begin{array}{l}
\psi =\sqrt{-1}\eta \left( X\right) , \\ 
\iota _{X}\omega ^{T}=\sqrt{-1}\overline{\partial }_{B}\psi ,%
\end{array}%
\right. 
\end{array}
\label{pr1}
\end{equation}%
it satisfies%
\begin{equation*}
\begin{array}{c}
Ric^{T}+\frac{1}{2}\mathcal{L}_{X}g^{T}=\left( A+2\right) g^{T}.%
\end{array}%
\end{equation*}%
Here $\iota _{X}$ denotes the contraction with $X$ and $\omega ^{T}=\frac{1}{%
2}d\eta =g^{T}\left( \Phi \left( \cdot \right) ,\cdot \right) .$ It is
called expanding, steady and shrinking if 
\begin{equation*}
A<-2;\text{\ }A=-2,\text{ }-2<A=2n
\end{equation*}%
respectively. It is called the gradient Sasaki-Ricci soliton if there is a
real Hamiltonian basic function $\psi $ with 
\begin{equation*}
\begin{array}{c}
\frac{1}{2}\mathcal{L}_{X}g^{T}=\psi _{j\overline{k}}%
\end{array}%
\end{equation*}%
such that 
\begin{equation*}
\begin{array}{c}
R_{j\overline{k}}^{T}+\psi _{j\overline{k}}=(A+2)g_{j\overline{k}}.%
\end{array}%
\end{equation*}%
or 
\begin{equation*}
\begin{array}{c}
R_{j\overline{k}}+\psi _{j\overline{k}}=Ag_{j\overline{k}}%
\end{array}%
\end{equation*}%
on $D_{p}.$
\end{defi}

Next we recall 

\begin{equation}\label{eqn-SR1}\aligned
R(X, Y)Z :=& \n_X\n_Y Z - \n_Y\n_X Z - \n_{[X, Y]}Z, \\
R(X, Y, Z, W) := & \la R(Z, W)Y, X \ra, \\
R_{\a \b \gamma \d}:=& R(e_\a, e_\b,e_\gamma, e_\d) = \la R(e_\gamma, e_\d)e_\b, e_\a \ra, \\
Ric(X, Y):=& \sum_{\a=1}^{2n+1} R(X, e_\a, Y, e_\a),
\endaligned
\end{equation}
where $\{ e_1, \cdot\cdot\cdot, e_{2n} \}$ is a local orthonormal frame on $\Gamma(D)$ and $e_{2n+1}=\xi$.

The second Bianchi identity is
\begin{equation}\label{eqn-SR1b}\aligned
R_{\a \b \gamma \d, \e}+ R_{\a\b \d \e, \gamma}+ R_{\a\b\e\gamma, \d}=0.
\endaligned
\end{equation}
By (\ref{eqn-Rm}) and Definition \ref{def1}, we have
\begin{equation}\label{eqn-SR3}\aligned
R_{i0j0}:= & R(e_i, \xi,e_j, \xi) = -R(e_i,\xi,\xi, e_j)  \\
=& -\la R(e_i, \xi)e_j, \xi \ra \\
=& - \la g(\xi, e_j)e_i - g(e_i, e_j)\xi, \xi \ra \\
=& g(e_i, e_j)=\d_{ij},
\endaligned
\end{equation}
\begin{equation}\label{eqn-SR4}\aligned
 R_{00}:= Ric(\xi, \xi) = \sum_{i=1}^{2n}R(e_i, \xi,e_i, \xi) + R_{0000}= 2n,
\endaligned
\end{equation}
and
\begin{equation}\label{eqn-SR4b}\aligned
R_{i0}:= & \sum_{\a=1}^{2n+1}R_{i\a 0\a} = \sum_{j=1}^{2n} R_{ij0j} + R_{i000}  \\
=& - \sum_{j=1}^{2n} R_{j0ij} = - \sum_{j=1}^{2n} \la R(e_j,\xi)e_j, e_i \ra \\
=& - \sum_{j=1}^{2n} \la g(\xi, e_j)e_j - g(e_j, e_j)\xi, e_i \ra \\
=& 0.
\endaligned
\end{equation}
Notice that
\begin{equation*}\label{eqn-SR5a}\aligned
R^T:=& \sum_{i=1}^{2n} Ric^T (e_i, e_i) \\
=& \sum_{i=1}^{2n} (Ric +2g)(e_i, e_i) \\
=& \sum_{i=1}^{2n} Ric(e_i, e_i) +\sum_{i=1}^{2n} 2g(e_i, e_i) \\
\endaligned
\end{equation*}
Then by using (\ref{eqn-SR4}), we obtain
\begin{equation}\label{eqn-SR5}\aligned
R^T=& \sum_{\a=1}^{2n+1}Ric(e_\a, e_\a) - Ric(\xi, \xi) +\sum_{i=1}^{2n} 2(g^T+\eta\otimes \eta)(e_i, e_i) \\
=& R - 2n + 4n = R+2n.
\endaligned
\end{equation}
From the above computation in (\ref{eqn-SR5}), we see that
\begin{equation}\label{eqn-SR6}\aligned
\sum_{i=1}^{2n}R_{ii} = \sum_{i=1}^{2n}  Ric(e_i, e_i) = R^T - 4n = R-2n.
\endaligned
\end{equation}
The gradient shrinking Sasaki-Ricci soliton equation implies
\begin{equation}\label{eqn-SR8}\aligned
R^T_{ij} -(2n+2)g_{ij} = -\psi_{ij}.
\endaligned
\end{equation}
Together with (\ref{eqn-SR2}), we get
\begin{equation}\label{eqn-SR9}\aligned
R_{ij} = 2n g_{ij} -\psi_{ij}
\endaligned
\end{equation}
and 
\begin{equation}\label{eqn-SR9b}\aligned
R = \sum_{j=1}^{2n}  R_{jj}+ R_{00} = 4n^2 - \D_B \psi +2n.
\endaligned
\end{equation}

\vskip24pt

\section{Proofs of main theorems}

The main theorems will be proven in this section. We firstly give the following lemma, which will play an important role in the proof of Theorem \ref{thm-SGR2}. 

\begin{lem}\label{lem1}

Let $(M^{2n+1}, g^T, \psi, X)$ be a complete gradient shrinking Sasaki-Ricci soliton with $\psi$ the corresponding real Hamiltonian basic function (potential) with respect to $X$, then we have
\begin{equation}\label{eqn-lem1}\aligned
\D_{B,\psi} R_{jl} := \D_B R_{jl} - \la \n\psi, \n R_{jl} \ra = 4n(R_{jl}-g_{jl}) -2\sum_{p, q=1}^{2n} R_{pq}R_{jplq}.
\endaligned
\end{equation}

\end{lem}

\begin{proof}
 
 From (\ref{eqn-SR1b}) and (\ref{eqn-SR3}), we derive
\begin{equation}\label{eqn-lem2}\aligned
\sum_{i=1}^{2n} R_{ijkl,i}=& \sum_{i=1}^{2n} R_{ijil,k} -\sum_{i=1}^{2n} R_{ijik,l} \\
=& \sum_{\a=1}^{2n+1} R_{\a j\a l, k}-R_{0j0l,k} - \left(\sum_{\a=1}^{2n+1} R_{\a j\a k, l} - R_{0j0k,l} \right) \\
=& (R_{jl}-g_{jl})_{,k} - (R_{jk}-g_{jk})_{,l} \\
=& R_{jl,k} - R_{jk,l}.
\endaligned
\end{equation}
 Then by (\ref{eqn-SR9}), (\ref{eqn-lem2}) and the Ricci identity, we get
 \begin{equation}\label{eqn-lem3}\aligned
\sum_{i=1}^{2n}  R_{ijkl, i} =&  R_{jl,k} - R_{jk, l} = \psi_{jk,l}-\psi_{jl,k} = \sum_{i=1}^{2n} \psi_i R_{ijkl}.
\endaligned
\end{equation}
The Ricci identity implies
 \begin{equation}\label{eqn-lem10}\aligned
R_{klpj,ip}  = R_{klpj,pi} + \sum_{q=1}^{2n} (R_{qlpj} R_{qkip} + R_{kqpj}R_{qlip} + R_{klqj}R_{qpip} + R_{klpq}R_{qjip}).
\endaligned
\end{equation}
Similarly, 
 \begin{equation}\label{eqn-lem11}\aligned
R_{klpi,jp}  = & R_{klpi,pj} + \sum_{q=1}^{2n} (R_{qlpi} R_{qkjp} + R_{kqpi}R_{qljp} + R_{klqi}R_{qpjp} + R_{klpq}R_{qijp}).
\endaligned
\end{equation}
By (\ref{eqn-lem10}), (\ref{eqn-lem11}) and the second Bianchi identity, we obtain
 \begin{equation}\label{eqn-lem12}\aligned
\D_B R_{ijkl} =&  \sum_{p=1}^{2n} R_{ijkl,pp} = \sum_{p=1}^{2n} (R_{klpj,ip} - R_{klpi,jp}) \\
=& \sum_{p=1}^{2n} (R_{klpj,pi} - R_{klpi,pj}) \\
&+ \sum_{p, q=1}^{2n} ( R_{qlpj}R_{qkip}+ R_{kqpj}R_{qlip}+ R_{klqj}R_{qpip} + R_{klpq}R_{qjip} ) \\
&- \sum_{p, q=1}^{2n} ( R_{qlpi}R_{qkjp}+ R_{kqpi}R_{qljp}+ R_{klqi}R_{qpjp} + R_{klpq}R_{qijp} ) \\
=:& I_1 + (I_2 + I_3 + I_4 + I_5 ) - (I_6 + I_7 + I_8 + I_9).
\endaligned
\end{equation}
From (\ref{eqn-lem3}), we derive
 \begin{equation}\label{eqn-lem13}\aligned
I_1 = & \sum_{p=1}^{2n} \left[ (R_{pjkl,p})_{,i} - (R_{klpi,p})_{,j} \right] \\
=& \sum_{p=1}^{2n} [ ( \psi_p R_{pjkl} )_{,i} - ( \psi_p R_{pikl} )_{,j} ] \\
=& \sum_{p=1}^{2n}  (R_{klpj,i}\psi_p - R_{klpi,j}\psi_p + R_{klpj}\psi_{pi} - R_{klpi}\psi_{pj}).
\endaligned
\end{equation}
By  (\ref{eqn-SR9}),
 \begin{equation}\label{eqn-lem14}\aligned
 R_{klpj}\psi_{pi} - R_{klpi}\psi_{pj} =& R_{klpj}(2ng_{pi}-R_{pi}) - R_{klpi} (2n g_{pj} - R_{pj}) \\
 = & 4n R_{klij} + R_{klpi}R_{pj} - R_{klpj}R_{pi}.
\endaligned
\end{equation}
Using the second Bianchi identity again,
 \begin{equation}\label{eqn-lem15}\aligned
R_{klpj,i}\psi_p - R_{klpi,j}\psi_p =& (R_{klpj,i} - R_{klpi,j}) \psi_p \\
=& R_{klij,p}\psi_p = R_{ijkl,p}\psi_p.
\endaligned
\end{equation}
Substituting (\ref{eqn-lem14}) and (\ref{eqn-lem15}) into (\ref{eqn-lem13}),
 \begin{equation}\label{eqn-lem16}\aligned
I_1 = & \sum_{p=1}^{2n} R_{ijkl,p}\psi_p + 4n R_{ijkl} + \sum_{p=1}^{2n} R_{klpi}R_{pj} - \sum_{p=1}^{2n} R_{klpj}R_{pi}.
\endaligned
\end{equation}
Direct computation gives us
 \begin{equation}\label{eqn-lem17}\aligned
I_2 - I_7 = 2 \sum_{p, q=1}^{2n} R_{qlpj}R_{qkip},  \quad I_3 - I_6 = 2\sum_{p, q=1}^{2n} R_{kqpj}R_{qlip}.
\endaligned
\end{equation}
The first Bianchi identity implies
 \begin{equation}\label{eqn-lem18}\aligned
I_5 - I_9 = \sum_{p, q=1}^{2n} R_{klpq}(R_{qjip}-R_{qijp}) = \sum_{p, q=1}^{2n} R_{klpq}R_{qpij}.
\endaligned
\end{equation}
By (\ref{eqn-SR3}),
 \begin{equation*}\label{eqn-lem19}\aligned
\sum_{p=1}^{2n} R_{qpip}= \sum_{\a=1}^{2n+1} R_{q\a i\a} - R_{q0i0} = R_{qi} - g_{qi}.
\endaligned
\end{equation*}
It follows
 \begin{equation}\label{eqn-lem20}\aligned
I_4 - I_8 = \sum_{q=1}^{2n}R_{klqj}(R_{qi} - g_{qi}) -  \sum_{q=1}^{2n} R_{klqi}(R_{qj} - g_{qj}).
\endaligned
\end{equation}
Substituting (\ref{eqn-lem16})--(\ref{eqn-lem20}) into (\ref{eqn-lem12}), 
 \begin{equation*}\label{eqn-lem21}\aligned
\D_B R_{ijkl} =& \sum_{p=1}^{2n} R_{ijkl,p}\psi_p + 4n R_{ijkl} + \sum_{p=1}^{2n} R_{klpi}R_{pj} - \sum_{p=1}^{2n} R_{klpj}R_{pi} \\
& + 2 \sum_{p, q=1}^{2n} R_{qlpj}R_{qkip} + 2\sum_{p, q=1}^{2n} R_{kqpj}R_{qlip} + \sum_{p, q=1}^{2n} R_{klpq}R_{qpij} \\
& + \sum_{q=1}^{2n} R_{klqj}(R_{qi} - g_{qi}) -  \sum_{q=1}^{2n} R_{klqi}(R_{qj} - g_{qj}).
\endaligned
\end{equation*}
Notice that
 \begin{equation*}\label{eqn-lem21b}\aligned
4nR_{ijkl} - \sum_{q=1}^{2n} R_{klqj}g_{qi} + \sum_{q=1}^{2n} R_{klqi}g_{qj} = (4n-2) R_{ijkl},
\endaligned
\end{equation*}
 \begin{equation*}\label{eqn-lem21c}\aligned
\sum_{p=1}^{2n} R_{klpi}R_{pj} - \sum_{q=1}^{2n} R_{klqi}R_{qj} = 0, \quad \sum_{q=1}^{2n} R_{klqj}R_{qi} - \sum_{p=1}^{2n} R_{klpj}R_{pi} = 0.
\endaligned
\end{equation*}
Thus we have
 \begin{equation*}\label{eqn-lem22}\aligned
\D_B R_{ijkl} =& \sum_{p=1}^{2n} R_{ijkl,p}\psi_p +(4n-2) R_{ijkl} \\
& + 2\sum_{p, q=1}^{2n} (R_{iplq}R_{jpkq} -  R_{ipkq}R_{jplq}) - \sum_{p, q=1}^{2n} R_{pqij}R_{pqkl}.
\endaligned
\end{equation*}
This implies
 \begin{equation}\label{eqn-lem23}\aligned
\D_{B,\psi} R_{ijkl} =& \D_B R_{ijkl} - \la \n \psi, \n R_{ijkl} \ra \\
=& \D_B R_{ijkl} - \sum_{p=1}^{2n} R_{ijkl,p}\psi_p  \\
=& (4n-2) R_{ijkl} 
 + 2\sum_{p, q=1}^{2n} (R_{iplq}R_{jpkq} - R_{ipkq}R_{jplq}) - \sum_{p, q=1}^{2n} R_{pqij}R_{pqkl}.
\endaligned
\end{equation}
From (\ref{eqn-lem23}), we get
\begin{equation*}\label{eqn-lem24}\aligned
\D_{B,\psi}\left( \sum_{i=1}^{2n} R_{ijil} \right)=& (4n-2) \sum_{i=1}^{2n} R_{ijil} \\
 & + 2\sum_{i, p, q=1}^{2n}(R_{iplq}R_{jpiq} - R_{ipiq}R_{jplq}) - \sum_{i, p, q=1}^{2n} R_{pqij}R_{pqil} \\
 =& (4n-2) (R_{jl} - g_{jl}) + 2 \sum_{i, p, q=1}^{2n} R_{iplq}R_{jpiq} \\
 & - 2 \sum_{p, q=1}^{2n} (R_{pq}-g_{pq})R_{jplq} - \sum_{i, p, q=1}^{2n} R_{pqij}R_{pqil}.
\endaligned
\end{equation*}
Notice that
\[
2\sum_{p, q=1}^{2n} g_{pq}R_{jplq}=2 \sum_{p=1}^{2n} R_{jplp}=2(R_{jl}-g_{jl}),
\]
and 
\[
\D_{B,\psi} R_{jl} = \D_{B,\psi}\left( \sum_{i=1}^{2n} R_{ijil} + R_{0j0l}\right) 
= \D_{B,\psi}\left( \sum_{i=1}^{2n} R_{ijil} \right).
\]
Then we obtain
\begin{equation}\label{eqn-lem25}\aligned
\D_{B,\psi} R_{jl} 
=& 4n (R_{jl} - g_{jl}) + 2 \sum_{i, p, q=1}^{2n} R_{iplq}R_{jpiq} \\
& - 2\sum_{p, q=1}^{2n} R_{pq}R_{jplq}- \sum_{i, p, q=1}^{2n} R_{ijpq}R_{pqil}.
\endaligned
\end{equation}
By the first Bianchi identity,
\begin{equation}\label{eqn-lem26}\aligned
 2 \sum_{i, p, q=1}^{2n} R_{iplq}R_{jpiq} - \sum_{i, p, q=1}^{2n} R_{ijpq}R_{pqil} =& 2\sum_{i, p, q=1}^{2n} R_{pilq} R_{jipq}+ \sum_{i, p, q=1}^{2n}R_{jipq} R_{pqil} \\
 =& \sum_{i, p, q=1}^{2n} R_{jipq}(2R_{pilq} + R_{pqil}) \\
 =& \sum_{i, p, q=1}^{2n} R_{jipq}(2R_{pilq} - R_{pilq} -R_{plqi} ) \\
 =&  \sum_{i, p, q=1}^{2n} (R_{jipq} R_{pilq} - R_{jipq}R_{plqi}).
\endaligned
\end{equation}
Notice that
\[
\sum_{i, p, q=1}^{2n} R_{jipq}R_{plqi} = \sum_{i, p, q=1}^{2n} R_{jiqp}R_{qlpi} = - \sum_{i, p, q=1}^{2n}R_{jipq} R_{piql}.
\]
Therefore from (\ref{eqn-lem26}), we have
\begin{equation}\label{eqn-lem27}\aligned
 2 \sum_{i, p, q=1}^{2n} R_{iplq}R_{jpiq} - \sum_{i, p, q=1}^{2n} R_{ijpq}R_{pqil} =  \sum_{i, p, q=1}^{2n} R_{jipq} (R_{pilq} + R_{piql})=0.
\endaligned
\end{equation}
Combining (\ref{eqn-lem25}) with (\ref{eqn-lem27}), it follows
\begin{equation*}\label{eqn-lem28}\aligned
\D_{B,\psi} R_{jl} =& 4n (R_{jl} - g_{jl}) - 2\sum_{p, q=1}^{2n} R_{pq}R_{jplq}.
\endaligned
\end{equation*}
\end{proof}

\noindent{\bf Proof of Theorem \ref{thm-SGR1}} \quad From  (\ref{eqn-SR6}) and (\ref{eqn-SR9}), we obtain
\begin{equation}\label{eqn-SR17}\aligned
\sum_{i,j=1}^{2n} R_{ij}^2 = \sum_{i,j=1}^{2n} R_{ij}(2n g_{ij}-\psi_{ij})  
=& 2n\sum_{i=1}^{2n} R_{ii} -\sum_{i,j=1}^{2n}R_{ij}\psi_{ij} \\
=& 2n(R-2n) -\sum_{i,j=1}^{2n} R_{ij}\psi_{ij}.
\endaligned
\end{equation}
Then by (\ref{eqn-SR4}), (\ref{eqn-SR4b}) and (\ref{eqn-SR17}),
\begin{equation}\label{eqn-SR18}\aligned
| Ric |^2=& \sum_{i,j=1}^{2n} R_{ij}^2 + 2\sum_{i=1}^{2n}R_{i0}^2 +R_{00}^2 \\
=&\sum_{i,j=1}^{2n} R_{ij}^2 +4n^2 
= 2nR -\sum_{i,j=1}^{2n} R_{ij}\psi_{ij}. 
\endaligned
\end{equation}

Let $\phi$ be a cut-off function on $M$. The above equality (\ref{eqn-SR18}) implies  
\begin{equation}\label{eqn-SR19a}\aligned
\int_M | Ric |^2 e^{-\l \psi}\phi^2=& 2n \int_M Re^{-\l \psi}\phi^2 - \sum_{i, j=1}^{2n} \int_M R_{ij}\psi_{ij}e^{-\l \psi}\phi^2 \\
=& 2n \int_M Re^{-\l \psi}\phi^2  + \sum_{i, j=1}^{2n} \int_M \psi_i \n_j(R_{ij}e^{-\l \psi}\phi^2) \\
=& 2n \int_M Re^{-\l \psi}\phi^2  + \sum_{i, j=1}^{2n} \int_M \psi_i \n_j(R_{ij}e^{-\psi})\cdot e^{(1-\l)\psi}\cdot \phi^2 \\
& + \sum_{i, j=1}^{2n}\int_M \psi_i R_{ij}e^{-\l\psi}\cdot (1-\l)\psi_j \phi^2 + \sum_{i, j=1}^{2n} \int_M \psi_i R_{ij}e^{-\l \psi}(\phi^2)_j.
 \endaligned
\end{equation}

By the equality (\ref{eqn-lem3}),
\begin{equation*}\label{eqn-SR10}\aligned
\sum_{j-1}^{2n} ( R_{jl,j} - R_{jj,l} ) = \sum_{i, j=1}^{2n} \psi_i R_{ijjl}.
\endaligned
\end{equation*}
It follows
\begin{equation*}\label{eqn-SR11a}\aligned
\sum_{j=1}^{2n}R_{jl,j}-(R-2n)_{,l}=& - \sum_{i=1}^{2n} \psi_i (R_{il}-R_{i0l0}) \\
=& -\sum_{i=1}^{2n} \psi_i R_{il}+\psi_l.
\endaligned
\end{equation*}
Namely,
\begin{equation}\label{eqn-SR11}\aligned
\sum_{j=1}^{2n} R_{jl,j}-R_{,l}= -\sum_{i=1}^{2n} \psi_i R_{il}+\psi_l.
\endaligned
\end{equation}
From (\ref{eqn-SR1b}), we derive
\begin{equation}\label{eqn-SR12a}\aligned
R_{ijkj,p}+ R_{ijjp,k}+ R_{ijpk,j}=0.
\endaligned
\end{equation}
Notice that
\begin{equation}\label{eqn-SR12b}\aligned
\sum_{j=1}^{2n} R_{ijkj,p}= &\sum_{\a=1}^{2n+1} R_{i\a k\a,p}-R_{i0k0,p} = R_{ik,p}, \\
\sum_{j=1}^{2n} R_{ijjp,k} =& - \sum_{\a=1}^{2n+1} R_{i\a p \a,k} + R_{i0p0, k} = - R_{ip,k}.
\endaligned
\end{equation}
Then by (\ref{eqn-SR12a}) and (\ref{eqn-SR12b}), we get
\begin{equation*}\label{eqn-SR12c}\aligned
R_{ik,p} -R_{ip, k} + \sum_{j=1}^{2n} R_{ijpk, j} = 0.
\endaligned
\end{equation*}
Thus 
\begin{equation*}\label{eqn-SR13}\aligned
 \sum_{i=1}^{2n} R_{ii,p}-  \sum_{i=1}^{2n} R_{ip,i} +  \sum_{i, j=1}^{2n} R_{ijpi,j}=0.
\endaligned
\end{equation*}
Namely,
\begin{equation}\label{eqn-SR14}\aligned
R_{,l} =  \sum_{i=1}^{2n} (R_{ii} + R_{00})_{, l} =  \sum_{i=1}^{2n} R_{ii,l}  = 2 \sum_{j=1}^{2n} R_{lj,j}.
\endaligned
\end{equation}
Combining (\ref{eqn-SR11}) with (\ref{eqn-SR14}), 
\begin{equation}\label{eqn-SR15}\aligned
R_{,l} = 2 \sum_{i=1}^{2n}  \psi_i R_{il}-2 \psi_l.
\endaligned
\end{equation}
Hence by (\ref{eqn-SR14}) and (\ref{eqn-SR15}), we have
\begin{equation}\label{eqn-SR16}\aligned
 \sum_{i=1}^{2n} \n_i(R_{ij}e^{-\psi}) =&  \sum_{i=1}^{2n} R_{ij, i}e^{-\psi} -  \sum_{i=1}^{2n} R_{ij}e^{-\psi}\psi_i \\
=&  e^{-\psi}\left( \sum_{i=1}^{2n}R_{ij,i} - \sum_{i=1}^{2n} \psi_i R_{ij} \right) \\
=& e^{-\psi}(\frac{1}{2}R_{,j} -  \sum_{i=1}^{2n} \psi_iR_{ij}) \\
=& -e^{-\psi}\psi_j.
\endaligned
\end{equation}
Sustituting (\ref{eqn-SR16}) into (\ref{eqn-SR19a}), we derive
\begin{equation}\label{eqn-SR19}\aligned
\int_M | Ric |^2 e^{-\l \psi}\phi^2
=& 2n \int_M Re^{-\l \psi}\phi^2 + (1-\l) \sum_{i, j=1}^{2n}\int_M  R_{ij}\psi_i \psi_j e^{-\l\psi} \phi^2 \\
& -  \int_M |\n\psi|^2 e^{-\l \psi}  \phi^2 + \sum_{i, j=1}^{2n} \int_M \psi_i R_{ij}e^{-\l \psi}(\phi^2)_j.
 \endaligned
\end{equation}
By (\ref{eqn-SR6}), (\ref{eqn-SR18}) and the Cauchy--Schwarz inequality, we obtain  
\begin{equation}\label{eqn-SR21a}\aligned
 2n \int_M Re^{-\l \psi}\phi^2 =& 2n \int_M (\sum_{i}R_{ii} + 2n)e^{-\l \psi}\phi^2 \\
 \leq & \frac{1}{8n} \int_M (\sum_i R_{ii})^2 e^{-\l \psi}\phi^2 + 8n^3 \int_M e^{-\l \psi}\phi^2 + 4n^2 \int_M e^{-\l \psi}\phi^2 \\
 \leq & \frac{1}{4} \int_M |Ric|^2 e^{-\l \psi}\phi^2 + (8n+3)n^2 \int_M e^{-\l \psi}\phi^2. 
\endaligned
\end{equation}
An easy algebraic manipulation gives
\begin{equation}\label{eqn-SR21}\aligned
&(1-\l) \sum_{i, j}\int_M  R_{ij}\psi_i \psi_j e^{-\l\psi} \phi^2 \\
\leq & \frac{1}{4}\int_M (| Ric |^2-4n^2) e^{-\l\psi}\phi^2 +(1-\l)^2 \int_M |\n \psi|^4 e^{-\l\psi}\phi^2, \\
\endaligned
\end{equation}
\begin{equation}\label{eqn-SR21c}\aligned
\sum_{i, j} \int_M \psi_i R_{ij}e^{-\l \psi}(\phi^2)_j \leq & \frac{1}{4}\int_M (| Ric |^2-4n^2) e^{-\l\psi}\phi^2 + 4 \int_M |\n \psi|^2 e^{-\l\psi}|\n \phi|^2.
\endaligned
\end{equation}
From (\ref{eqn-SR19})--(\ref{eqn-SR21c}), we have
\begin{equation}\label{eqn-SR21d}\aligned
\int_M | Ric |^2 e^{-\l \psi}\phi^2
\leq &  4(8n+1)n^2 \int_M e^{-\l \psi}\phi^2 \\
& + 4(1-\l)^2 \int_M |\n \psi|^4 e^{-\l \psi}\phi^2 
 +16 \int_M |\n \psi|^2 e^{-\l\psi}|\n \phi|^2.
 \endaligned
\end{equation}
It follows from Lemma 2 in \cite{CLL25} by assuming $C_1=0$, we have 
\begin{equation}\label{eqn-SR22}\aligned
R+|\n\psi|^2=(4n-2)\psi.
\endaligned
\end{equation}
By proposition 3 in \cite{CLL25}, we get
\begin{equation}\label{eqn-SR23}\aligned
n(d(x, y)-7)^2_{+} \leq \psi(x)+C_2 \leq n(d(x, y)+\sqrt 3)^2,
\endaligned
\end{equation}
where $y$ is a minimum point of $\psi$.

 By using (\ref{eqn-SR23}), for any $\mu>0$,
\begin{equation}\label{eqn-SGRS11}\aligned
\int_M   e^{-\mu\psi} =& \sum_{j=0}^{\infty} \int_{B_{(j+1)r}\backslash B_{j r}} e^{-\mu \psi}  \\
\leq & \sum_{j=0}^\infty e^{-\mu n(jr-7)^2} \cdot {\rm Vol}(B_{(j+1)r}) \\
\leq & \sum_{j=0}^\infty Ce^{-\mu n(jr-7)^2} (j+1)^d r^d < \infty.
\endaligned
\end{equation}
The formula (6.1)(v) in \cite{CLL25} implies that the scalar curvature $R$ satisfies
\begin{equation}\label{eqn-SR24}\aligned
R \geq \frac{C_3}{\psi}
\endaligned
\end{equation}
for some positive constant $C_3$. From (\ref{eqn-SR22})--(\ref{eqn-SR24}), we know that
\[
\int_M |\n\psi|^4 e^{-\l \psi} < \infty \quad and \quad \int_M |\n \psi|^2 e^{-\l \psi} < \infty.
\]
Hence by (\ref{eqn-SR21d}), we conclude that
\[
\int_M | Ric |^2 e^{-\l \psi} < \infty.
\]
\qed

\noindent{\bf Proof of Theorem \ref{thm-SGR2}} \quad By (\ref{eqn-lem3}),  we obtain
\begin{equation}\label{eqn-SRS3}\aligned
\sum_{i=1}^{2n} \n_i \left(R_{ijkl} e^{-\psi} \right) = e^{-\psi}\left( \sum_{i=1}^{2n} R_{ijkl, i}-\sum_{i=1}^{2n} \psi_i R_{ijkl}\right) = 0.
\endaligned
\end{equation}
The second Bianchi identity implies
\begin{equation}\label{eqn-SRS3b}\aligned
R_{ij,0} =& \sum_{\a=1}^{2n+1} R_{i\a j\a, 0} = \sum_{k=1}^{2n} R_{ikjk,0}+ R_{i0j0,0} \\
=&  \sum_{k=1}^{2n}R_{ikjk, 0} = -  \sum_{k=1}^{2n} (R_{ikk0, j} + R_{ik0j, k}) \\
=& 0.
\endaligned
\end{equation}
Combining (\ref{eqn-SR4}), (\ref{eqn-SR4b}), (\ref{eqn-SR9}) with (\ref{eqn-SRS3b}), it follows
\begin{equation}\label{eqn-SRS5}\aligned
|\n Ric|^2 = &\sum_{\a,\b,\gamma=1}^{2n+1} |\n_\gamma R_{\a\b}|^2 \\
=& \sum_{\gamma=1}^{2n+1} \sum_{i,j=1}^{2n} |\n_\gamma R_{ij}|^2 + 2 \sum_{\gamma=1}^{2n+1}\sum_{i=1}^{2n} |\n_\gamma R_{i0}|^2 + \sum_{\gamma=1}^{2n+1} |\n_\gamma R_{00}|^2 \\
=& \sum_{\gamma=1}^{2n+1} \sum_{i,j=1}^{2n} |\n_\gamma R_{ij}|^2 \\
=&  \sum_{i,j,k=1}^{2n} |\n_k R_{ij}|^2 = \sum_{i, j=1}^{2n} |\n^T R_{ij}|^2.
\endaligned
\end{equation}

Let $B_r$ be the closed geodesic ball of $M$. Let $\phi$ be a smooth cut-off function on $M$ such that $\phi \equiv 1$ on $B_r, \phi \equiv 0$ outside $B_{2r}$ and $|\n\phi \leq \frac{C}{r}$ on $B_{2r} \backslash B_r$. 
Notice that
\begin{equation*}\label{eqn-SRS6}\aligned
\D_{B,\psi}R_{ij} = \D_B R_{ij} - \la \n \psi, \n R_{ij} \ra = e^\psi \div^T\left( e^{-\psi} \n^T R_{ij} \right),
\endaligned
\end{equation*}
and 
\begin{equation}\label{eqn-SRS7}\aligned
&\div^B\left( e^{-\psi} \n^T R_{ij} \cdot R_{ij} \phi^2 \right) \\
=& \div^B\left( e^{-\psi} \n^T R_{ij} \right) R_{ij}\phi^2 + \la e^{-\psi}\n^T R_{ij}, \n^T(R_{ij}\phi^2) \ra \\
=& (\D_{B, \psi}R_{ij}) R_{ij} e^{-\psi} \phi^2 + e^{-\psi} |\n^T R_{ij}|^2 \phi^2 + e^{-\psi} R_{ij} \la \n^T R_{ij}, \n^T \phi^2 \ra.
\endaligned
\end{equation}
Then from (\ref{eqn-SRS5}) and (\ref{eqn-SRS7}), we get
\begin{equation*}\label{eqn-SRS8a}\aligned
\int_M |\n Ric|^2 e^{-\psi} \phi^2
= - \sum_{i,j=1}^{2n} \int_M (\D_{B, \psi}R_{ij})R_{ij} e^{-\psi}\phi^2 -  \sum_{i,j,k=1}^{2n} \int_M (\n_k R_{ij})R_{ij}e^{-\psi}(\phi^2)_k. 
\endaligned
\end{equation*}
Hence by Lemma \ref{lem1}, we obtain
\begin{equation}\label{eqn-SRS8}\aligned
&\int_M |\n Ric|^2 e^{-\psi} \phi^2 \\
=& -  \sum_{i,j=1}^{2n} \int_M \left( 4n (R_{ij}-g_{ij} ) -2  \sum_{p,q=1}^{2n} R_{pq}R_{ipjq} \right) R_{ij} e^{-\psi} \phi^2 \\
& -  \sum_{i,j,k=1}^{2n}\int_M (\n_k R_{ij}) R_{ij} e^{-\psi} (\phi^2)_k \\
=& -4n  \sum_{i,j=1}^{2n} \int_M R_{ij}^2 e^{-\psi} \phi^2  + 4n  \sum_{i,j=1}^{2n} \int_M g_{ij}R_{ij} e^{-\psi} \phi^2 \\ 
& + 2  \sum_{i,j,p,q=1}^{2n} \int_M R_{ipjq} R_{ij}R_{pq} e^{-\psi} \phi^2 -  \sum_{i,j,k=1}^{2n}\int_M (\n_k R_{ij}) R_{ij} e^{-\psi} (\phi^2)_k.
\endaligned
\end{equation}
The equalities (\ref{eqn-SR3}) and (\ref{eqn-SR9}) imply that
\begin{equation*}\label{eqn-SRS9a}\aligned
 \sum_{i,j,p,q=1}^{2n} R_{ipjq} R_{ij}R_{pq} = &  \sum_{i,j,p,q=1}^{2n} R_{ipjq} R_{ij}(2ng_{pq}-\psi_{pq}) \\
 =& 2n  \sum_{i,j,p=1}^{2n} R_{ij}R_{ipjp} -  \sum_{i,j,p,q=1}^{2n} R_{ipjq}R_{ij}\psi_{pq} \\
 =& 2n  \sum_{i,j=1}^{2n} R_{ij}(R_{ij}-g_{ij}) -  \sum_{i,j,p,q=1}^{2n} R_{ipjq}R_{ij}\psi_{pq} \\
 =& 2n \sum_{i,j=1}^{2n} R_{ij}^2 - 2n  \sum_{i,j=1}^{2n} g_{ij}R_{ij} -  \sum_{i,j,p,q=1}^{2n} R_{ipjq}R_{ij}\psi_{pq}.
\endaligned
\end{equation*}
It follows 
\begin{equation}\label{eqn-SRS9}\aligned
& 2 \sum_{i,j,p,q=1}^{2n} \int_M R_{ipjq} R_{ij}R_{pq} e^{-\psi} \phi^2 \\
=& 4n\sum_{i,j=1}^{2n}  \int_M R_{ij}^2 e^{-\psi} \phi^2 - 4n \sum_{i,j=1}^{2n}  \int_M g_{ij}R_{ij} e^{-\psi} \phi^2 \\
& - 2 \sum_{i,j,p,q=1}^{2n} \int_M R_{ipjq}R_{ij} \psi_{pq}e^{-\psi}\phi^2.
\endaligned
\end{equation}
From (\ref{eqn-SRS8}) and (\ref{eqn-SRS9}), we derive
\begin{equation}\label{eqn-SRS10}\aligned
&\int_M |\n Ric|^2 e^{-\psi} \phi^2 \\
=& - 2\sum_{i,j,p,q=1}^{2n} \int_M R_{ipjq}R_{ij} \psi_{pq}e^{-\psi}\phi^2 \\
&- \sum_{i,j,k=1}^{2n} \int_M (\n_k R_{ij}) R_{ij} e^{-\psi} (\phi^2)_k.
\endaligned
\end{equation}
Integrating by parts and the equality (\ref{eqn-SRS3}) give us
\begin{equation}\label{eqn-SRS11}\aligned
 &- 2\sum_{i,j,p,q=1}^{2n} \int_M R_{ipjq}R_{ij} \psi_{pq}e^{-\psi}\phi^2 \\
 =& 2 \sum_{i,j,p,q=1}^{2n} \int_M \n_q (R_{qjpi}e^{-\psi}) \cdot R_{ij} \phi^2 \psi_p + 2\sum_{i,j,p,q=1}^{2n} \int_M \n_q(R_{ij}\phi^2) \cdot R_{ipjq} e^{-\psi}\psi_p \\
 =& 2\sum_{i,j,p,q=1}^{2n} \int_M R_{ipjq} (\n_q R_{ij})\psi_p e^{-\psi}\phi^2 + 2\sum_{i,j,p,q=1}^{2n} \int_M R_{ipjq}R_{ij} \psi_p e^{-\psi} (\phi^2)_q. 
\endaligned
\end{equation}
By (\ref{eqn-lem3}), we get
\begin{equation}\label{eqn-SGRS9}\aligned
(\div Rm)_{\b\gamma\d} =& \sum_{\a=1}^{2n+1} R_{\a\b\gamma\d, \a}= \sum_{i=1}^{2n} R_{i\b\gamma\d, i} + R_{0\b\gamma\d, 0} \\
=& \sum_{i=1}^{2n} R_{i\b\gamma\d, i} =  \sum_{i=1}^{2n} R_{ijkl, i}=(\div Rm)_{jkl} = \sum_{i=1}^{2n}  \psi_i R_{ijkl}.
\endaligned
\end{equation}
Direct computation gives
\begin{equation}\label{eqn-SRS12}\aligned
& 2\sum_{i,j,p,q=1}^{2n} R_{ipjq} (\n_q R_{ij})\psi_p = -2 \sum_{i,j,p,q=1}^{2n} R_{qjip}\psi_p (\n_q R_{ij}) \\
=& -2\sum_{i,j,p,q=1}^{2n} R_{jqip}\psi_p (\n_j R_{iq}) 
= 2 \sum_{i,j,p,q=1}^{2n} R_{qjip}\psi_p (\n_j R_{iq}) \\
=&  \sum_{i,j,p,q=1}^{2n} R_{qjip}\psi_p (\n_j R_{iq} - \n_q R_{ij}).
\endaligned
\end{equation}
Then from (\ref{eqn-lem3}), (\ref{eqn-SGRS9}) and (\ref{eqn-SRS12}), we deduce
\begin{equation}\label{eqn-SRS12b}\aligned
 2\sum_{i,j,p,q=1}^{2n} R_{ipjq} (\n_q R_{ij})\psi_p 
=& \sum_{i,j,p,q, h=1}^{2n} R_{qjip}\psi_p \cdot \psi_h R_{hijq} \\
=& \sum_{i,j,q=1}^{2n}\left( (\div Rm)_{ijq} \right)^2\\
=& \sum_{\b,\gamma,\d=1}^{2n+1} ((\div Rm)_{\b\gamma\d})^2 =  |\div Rm|^2.
\endaligned
\end{equation}
Substituting  (\ref{eqn-SRS11}) and (\ref{eqn-SRS12b}) into (\ref{eqn-SRS10}),
\begin{equation}\label{eqn-SRS13}\aligned
\int_M |\n Ric|^2 e^{-\psi} \phi^2 = & \int_M |\div Rm|^2 e^{-\psi}\phi^2 + 2\sum_{i,j,p,q=1}^{2n} \int_M R_{ipjq}R_{ij} \psi_p e^{-\psi} (\phi^2)_q \\
& - \sum_{i,j,k=1}^{2n}\int_M (\n_k R_{ij}) R_{ij} e^{-\psi} (\phi^2)_k.
\endaligned
\end{equation}
Notice that
\begin{equation*}\label{eqn-SRS16}\aligned
 \int_M |\div Rm|^2 e^{-\psi}\phi^2 \leq & C\int_M |Rm|^2 |\n \psi|^2 e^{-\psi} \\
 \leq & C\int_M |Rm|^2\psi e^{-\psi} \leq C\int_M |Rm|^2 e^{-\l \psi} < \infty,
\endaligned
\end{equation*}
and 
\begin{equation*}\label{eqn-SRS17}\aligned
 & 2\sum_{i,j,p,q=1}^{2n} \int_M R_{ipjq}R_{ij} \psi_p e^{-\psi} (\phi^2)_q \\
 \leq & C\int_M |Rm|^2 |\n \psi| e^{-\psi} \leq C\int_M |Rm|^2 e^{-\l \psi} < \infty.
\endaligned
\end{equation*}
Since
\begin{equation*}\label{eqn-SRS18}\aligned
\int_M |\n Ric|^2 e^{-\psi} \phi^2 \leq & C -\sum_{i,j,k=1}^{2n} \int_M (\n_k R_{ij}) R_{ij} e^{-\psi} (\phi^2)_k \\
 \leq & C+ 2\sum_{i,j,k=1}^{2n}\int_M |\n_k R_{ij}||R_{ij}|e^{-\psi} \phi |\n \phi| \\
\leq & C + \frac{1}{2}\int_M |\n Ric|^2 e^{-\psi}  \phi^2 + 2 \int_M |Ric|^2 e^{-\psi} |\n \phi|^2,
\endaligned
\end{equation*}
therefore by using Theorem \ref{thm-SGR1} again, we can conclude that
\[
\int_M |\n Ric|^2 e^{-\psi} < \infty.
\]
By H\"older inequality, we derive that as $r \to \infty$
\begin{equation}\label{eqn-SRS19}\aligned
 &\left| \sum_{i,j,k=1}^{2n}\int_M (\n_k R_{ij}) R_{ij} e^{-\psi} (\phi^2)_k \right| \\
 \leq & \frac{C}{r}\left( \int_M |\n Ric |^2 e^{-\psi} \right)^{\frac{1}{2}} \left( \int_{B_{2r}\backslash B_r} | Ric |^2e^{-\psi}\right)^{\frac{1}{2}} \to 0, 
\endaligned
\end{equation}
and
\begin{equation}\label{eqn-SRS20}\aligned
 \left| \sum_{i,j,p,q=1}^{2n}\int_M R_{ipjq}R_{ij} \psi_p e^{-\psi} (\phi^2)_q \right| \leq \frac{C}{r} \int_{B_{2r}\backslash B_r} | Rm |^2e^{-\l\psi} \to 0.
\endaligned
\end{equation}
Combining (\ref{eqn-SRS13}), (\ref{eqn-SRS19}) with (\ref{eqn-SRS20}), it follows 
\[
\int_M |{\n Ric}|^2 e^{-\psi} = \int_M |\div Rm|^2 e^{-\psi} < \infty.
\]
\qed

\begin{cor}\label{cor-SGR1}

Let $(M^{2n+1}, g, \psi)$ be a complete Sasakian gradient Ricci soliton with harmonic Weyl tensor.
  Then we have
\begin{equation}\label{eqn-SR20}\aligned
\int_M |{\n Ric}|^2 e^{-\psi} = \int_M |{\rm div} Rm|^2 e^{-\psi} < \infty.
\endaligned
\end{equation}

\end{cor}

\begin{proof}

For any $X, Y, Z\in \Gamma(D)$,
\begin{equation*}\aligned
(\n_X {\rm Hess \psi})(Y, Z) =& \n_X ({\rm Hess}\psi (Y, Z)) - {\rm Hess}\psi (\n_X Y, Z) - {\rm Hess}\psi (Y, \n_X Z), \\
(\n_Y {\rm Hess \psi})(X, Z) =& \n_Y ({\rm Hess}\psi (X, Z)) - {\rm Hess}\psi (\n_Y X, Z) - {\rm Hess}\psi (X, \n_Y Z).
\endaligned
\end{equation*}
Thus we obtain
\begin{equation}\label{eqn-SGRS1}\aligned
&(\n_X {\rm Hess \psi})(Y, Z) - (\n_Y {\rm Hess \psi})(X, Z) \\
=& \n_X ({\rm Hess}\psi (Y, Z)) - \n_Y ({\rm Hess}\psi (X, Z)) + {\rm Hess}\psi (\n_Y X -\n_X Y, Z) \\
& + {\rm Hess}\psi (X, \n_Y Z)- {\rm Hess}\psi (Y, \n_X Z).
\endaligned
\end{equation}
Notice that
\begin{equation}\label{eqn-SGRS2}\aligned
 \n_X ({\rm Hess}\psi (Y, Z)) =& \n_X \la \n_Y \n \psi, Z \ra = \la \n_X\n_Y \n \psi, Z \ra + \la \n_Y \n \psi, \n_X Z \ra, \\
  \n_Y ({\rm Hess}\psi (X, Z)) =& \n_Y \la \n_X \n \psi, Z \ra = \la \n_Y\n_X \n \psi, Z \ra + \la \n_X \n \psi, \n_Y Z \ra,
\endaligned
\end{equation}
and 
\begin{equation}\label{eqn-SGRS3}\aligned
{\rm Hess}\psi (X, \n_Y Z)- {\rm Hess}\psi (Y, \n_X Z) = \la \n_X \n \psi, \n_Y Z \ra - \la \n_Y \n \psi, \n_X Z \ra.
\endaligned
\end{equation}
Substituting (\ref{eqn-SGRS2}) and (\ref{eqn-SGRS3}) into (\ref{eqn-SGRS1}),
\begin{equation}\label{eqn-SGRS4}\aligned
(\n_X {\rm Hess \psi})(Y, Z) - (\n_Y {\rm Hess \psi})(X, Z) 
= \la R(X, Y)\n \psi, Z \ra. 
\endaligned
\end{equation}
If the Weyl tensor is harmonic, then the Schouten tensor 
\[
S= Ric - \frac{R}{4n}g
\]
is a Codazzi tensor. It yields
\[
(\n_X S)(Y, Z) = (\n_Y S) (X, Z).
\]
Note that
\[
S= Ric - \frac{R}{4n}g= Ric^T -2g -  \frac{R}{4n}g.
\]
Thus we derive
\[
\n_X S= \n_X Ric^T -  \frac{\n_X R}{4n}g.
\]
It follows that
\begin{equation}\label{eqn-SGRS5}\aligned
(\n_X  Ric^T)(Y, Z) - (\n_Y  Ric^T)(X, Z) = \frac{\n_X R}{4n}g(Y, Z) - \frac{\n_Y R}{4n}g(X, Z). 
\endaligned
\end{equation}
From the Sasakian gradient Ricci soliton equation, we get
\[
{Ric}^T + {\rm Hess}\psi = (2n+2)g^T.
\]
By (\ref{eqn-SGRS4}) and (\ref{eqn-SGRS5}) and the above equality, we have
 \begin{equation}\label{eqn-SGRS6}\aligned
&(2n+2)(\n_X g^T)(Y, Z) - (2n+2)(\n_Y g^T)(X, Z) \\
=&  \frac{\n_X R}{4n}g(Y, Z) - \frac{\n_Y R}{4n}g(X, Z)  + \la R(X, Y)\n \psi, Z \ra. 
\endaligned
\end{equation}
Direct computation gives
\[
\n_X g^T = \n_X g - \n_X (\eta\otimes \eta) = 0.
\]
Hence (\ref{eqn-SGRS6}) becomes 
\begin{equation}\label{eqn-SGRS7}\aligned
R(X, Y, Z, \n \psi) =  \la R(X, Y)\n \psi, Z \ra = \frac{\n_Y R}{4n}g(X, Z) -   \frac{\n_X R}{4n}g(Y, Z).
\endaligned
\end{equation}
Combining  (\ref{eqn-SR15}) with (\ref{eqn-SGRS7}), 
\begin{equation}\label{eqn-SGRS8}\aligned
& R(X, Y, Z, \n \psi)  \\
= &\frac{ Ric(Y, \n \psi)- g(Y, \n \psi)}{2n}g(X, Z)- \frac{ Ric(X, \n \psi)- g(X, \n \psi)}{2n}g(Y, Z).
\endaligned
\end{equation}
Taking $Z = \n\psi$, we obtain
\[
 Ric(Y, \n\psi) g(X, \n \psi) =  Ric(X, \n\psi) g(Y, \n \psi).
\]
If we consider $Y \bot \n \psi$, then by (\ref{eqn-SR9}), for every $X \in \Gamma(D)$, 
\[
0= Ric(Y, \n \psi)g(X, \n \psi) = {\rm Hess}(\psi)(Y, \n \psi)g(X, \n \psi). 
\]
Thus $\n\psi$ is an eigenvalue of $ Ric$ and ${\rm Hess}(\psi)$.

From (\ref{eqn-SGRS11}), (\ref{eqn-SGRS9}), (\ref{eqn-SGRS8}) and Theorem \ref{thm-SGR1}, we derive
\begin{equation*}\label{eqn-SGRS10}\aligned
\int_M |\div Rm|^2 e^{-\psi} \leq & C \int_M | Ric |^2 |\n \psi|^2 e^{-\psi} + C\int_M |\n\psi|^2 e^{-\psi} \\
\leq & C \int_M | Ric |^2  e^{-\mu\psi} + C\int_M  \psi e^{-\psi} \\
\leq & C+ C\int_M   e^{-\mu\psi} < \infty,
\endaligned
\end{equation*}
where $\mu \in (0, 1)$ is a constant. 
Moreover, we obtain as $r \to \infty$
\begin{equation}\label{eqn-SGRS13}\aligned
& \sum_{i,j,p,q=1}^{2n} \int_M R_{ipjq}R_{ij} \psi_p e^{-\psi} (\phi^2)_q \\
\leq & \frac{c}{r} \left( \int_M |\div Rm|^2 e^{-\psi} + \int_M | Ric |^2 e^{-\psi} \right) \leq \frac{C}{r} \to 0.
\endaligned
\end{equation}
Then combining (\ref{eqn-SRS13}), (\ref{eqn-SRS19}) with (\ref{eqn-SGRS13}), we have
\[
\int_M |{\n  Ric}|^2 e^{-\psi} = \int_M |\div Rm|^2 e^{-\psi} < \infty.
\]
\end{proof}

\noindent{\bf Proof of Theorem \ref{thm-main}} \quad
Let $\{E_1, \cdot\cdot\cdot, E_{2n}\} \subset \Gamma(D)$ be the eigenvectors of $ Ric$ with $\la E_i, E_j \ra = \d_{ij}$ and $E_{2n} = \frac{\n \psi}{|\n \psi|}$. 
Then employing (\ref{eqn-lem3}) and (\ref{eqn-SGRS8}), we get
\begin{equation}\label{eqn-SGRS14}\aligned
|\div Rm|^2 = & \sum_{j,k, l=1}^{2n} |(\div Rm)(E_j, E_k, E_l)|^2 = \sum_{j,k, l=1}^{2n}|R(\n\psi, E_j, E_k, E_l)|^2 \\
=& \sum_{k,l=1}^{2n} |R(\n\psi, E_k, E_k, E_l)|^2 + \sum_{k, l=1}^{2n} |R(\n\psi, E_l, E_k, E_l)|^2 \\
&+ \sum_{k, l=1}^{2n} \sum_{j\neq k, j\neq l}  |R(\n\psi, E_j, E_k, E_l)|^2 \\
=& 2  \sum_{k, l=1}^{2n} |R(\n\psi, E_k, E_k, E_l)|^2 \\
=& \frac{1}{n} \sum_{l=1}^{2n} |  Ric(E_l, \n\psi) - g(E_l, \n\psi) |^2 \\
=& \frac{1}{4n} |\n R|^2.
\endaligned
\end{equation}
Applying the Schwarz inequality and (\ref{eqn-SRS5}), 
\begin{equation}\label{eqn-SGRS15}\aligned
|\n Ric|^2 =& \sum_{i, j, k=1}^{2n} (R_{ij, k})^2 \geq \sum_{i, k=1}^{2n} (R_{ii, k})^2 \\
\geq & \frac{1}{2n}\sum_{k=1}^{2n} \left(\sum_{i=1}^{2n}R_{ii, k}\right)^2 = \frac{1}{2n}\sum_{k=1}^{2n}\left[(R-2n)_{, k}\right]^2 \\
=& \frac{1}{2n} |\n R|^2.
\endaligned
\end{equation}
From (\ref{eqn-SR20}),  (\ref{eqn-SGRS14}) and (\ref{eqn-SGRS15}), we get the inequality 
\[
\frac{1}{4n} \int_M |\n R|^2 e^{-\psi} \geq \frac{1}{2n} \int_M |\n R|^2 e^{-\psi}.
\]
This forces that $R$ is constant. 
Therefore by (\ref{eqn-SR15}), we derive
\begin{equation}\label{eqn-SGRS16}\aligned
 Ric(\n \psi, \n \psi) - |\n \psi|^2 = \frac{1}{2}\la \n R, \n \psi \ra = 0.
\endaligned
\end{equation}
Then from (\ref{eqn-SGRS8}) and (\ref{eqn-SGRS16}), we get that the transversely radial curvature
\begin{equation*}\label{eqn-SGRS17}\aligned
\kappa^T_{rad} =& \sum_{i=1}^{2n}R(E_i, \n \psi, E_i, \n \psi) \\
=& \sum_{i=1}^{2n-1}\frac{1}{2n}\left[  Ric(\n \psi, \n \psi) - |\n \psi|^2 \right]g(E_i, E_i) \\
=& \frac{2n-1}{2n}\left[  Ric(\n \psi, \n \psi) - |\n \psi|^2 \right] = 0.
\endaligned
\end{equation*}
Then by Theorem 1 in \cite{CLL25}, $(M,g,\psi )$ is transversely rigid.  Therefore $M$ is Sasaki-Einstein and compact shrinking Ricci soliton. Then by Theorem 2.5 in \cite{MS13}, $M$ is a finite quotient of $\mathbb{S}^{2n+1}$.
\qed

\vskip24pt

\noindent{\bf Acknowledgements} SC is partially supported by Startup Foundation for Advanced Talents of  the Shanghai Institute for Mathematics and Interdisciplinary Sciences (No.2302-SRFP-2024-0049). 
 HQ is partially supported by NSFC (No.12471050) and Hubei Provincial Natural Science Foundation of China (No.2024AFB746).

\end{document}